\documentclass{mcom-l}

\newtheorem{theorem}{Theorem}[section]
\newtheorem{lemma}[theorem]{Lemma}
\newtheorem{proposition}[theorem]{Proposition}
\newtheorem{corollary}[theorem]{Corollary}

\theoremstyle{definition}
\newtheorem{definition}[theorem]{Definition}
\newtheorem{example}[theorem]{Example}
\newtheorem{notation}[theorem]{Notation}

\theoremstyle{remark}
\newtheorem{remark}[theorem]{Remark}

\numberwithin{equation}{section}
\usepackage{enumitem}
\usepackage{bm}
\usepackage[dvipsnames]{xcolor}
\usepackage[colorlinks=true,linkcolor=RoyalBlue,citecolor=PineGreen,urlcolor=blue]{hyperref}
\usepackage{amsmath, amssymb}
\usepackage{algorithm}
\usepackage{multicol}

\DeclareMathOperator{\GL}{GL}
\DeclareMathOperator{\SL}{SL}
\DeclareMathOperator{\Mat}{Mat}
\DeclareMathOperator{\Lie}{Lie}

\DeclareMathOperator{\Ker}{Ker}

\makeatletter
\renewenvironment{cases}[1][l]{\matrix@check\cases\env@cases{#1}}{\endarray\right.}
\def\env@cases#1{%
  \let\@ifnextchar\new@ifnextchar
  \left\lbrace\def\arraystretch{1.2}%
  \array{@{}#1@{\quad}l@{}}}
\makeatother

\makeatletter
\renewcommand{\paragraph}{%
  \@startsection{paragraph}{4}%
  {\z@}{1.5ex \@plus 1ex \@minus .2ex}{-1em}%
  {\normalfont\normalsize\bfseries}%
}
\makeatother

\makeatletter
\let\orgdescriptionlabel\descriptionlabel
\renewcommand*{\descriptionlabel}[1]{%
  \let\orglabel\label
  \let\label\@gobble
  \phantomsection
  \edef\@currentlabel{#1}%
  \let\label\orglabel
  \orgdescriptionlabel{#1}%
}
\makeatother

\begin{document}

\title{Degree bound for toric envelope of a linear algebraic group}


\author{Eli Amzallag}
\address{Department of Mathematics, CUNY Graduate Center, 365 Fifth Avenue, New York, USA}
\curraddr{Department of Mathematics, The City College of New York, New York, USA}
\email{eamzallag@gradcenter.cuny.edu}

\author{Andrei Minchenko}
\address{Faculty of Mathematics, University of Vienna, Oskar-Morgenstern-Platz 1, Vienna, Austria}
\email{an.minchenko@gmail.com}

\author{Gleb Pogudin}
\address{Courant Institute of Mathematical Sciences, New York University, 251 Mercer st., New York,  USA}
\curraddr{LIX, CNRS, \'Ecole Polytechnique, Institute Polytechnique de Paris, Palaiseau, France}
\email{gleb.pogudin@polytechnique.edu}

\thanks{This work was partially supported by the NSF grants CCF-095259, CCF-1564132, CCF-1563942, DMS-1606334, DMS-1760448 by the NSA grant \#H98230-15-1-0245, by CUNY CIRG \#2248, by PSC-CUNY grants \#69827-00 47, \#60098-00 48, by the Austrian Science Fund FWF grants Y464-N18, P28079-N35.}


\date{}

\begin{abstract}
  Algorithms working with linear algebraic groups often represent them via defining polynomial equations.
  One can always choose defining equations for an algebraic group to be of degree at most the degree of the group as an algebraic variety.
  However, the degree of a linear algebraic group $G \subset \GL_n(C)$ can be arbitrarily large even for $n = 1$.
  One of the 
  key ingredients 
  of Hrushovski's algorithm for computing the Galois group of a linear differential equation was an idea to ``approximate'' every algebraic subgroup of $\GL_n(C)$ by a ``similar'' group so that the degree of the latter is bounded uniformly in $n$. Making this uniform bound computationally feasible is crucial for making the algorithm practical.
  
  In this paper, we derive a single-exponential degree bound for such an approximation (we call it a \emph{toric envelope}), which is qualitatively optimal.
  As an application, we improve the quintuply exponential bound due to Feng for the first step of Hrushovski's algorithm to a single-exponential bound.
  For the cases $n = 2, 3$ often arising in practice, we further refine our general bound.
\end{abstract}

\maketitle

\section{Introduction}

\subsection{Representing linear algebraic groups in algorithms}
A linear algebraic group is a subgroup of the group $\GL_n(C)$ of invertible $n\times n$ matrices over a field $C$ that is defined by a system of polynomial equations in matrix entries.
Such groups arise naturally in different areas of mathematics.

For algorithms dealing with arbitrary linear algebraic groups, there are two standard ways of representing such a group~\cite[Section~3.13]{deGraaf} 
\begin{enumerate}[label=\textbf{(R\arabic*)}]
  \item\label{rep:polys} by a system of defining polynomial equations 
  \item\label{rep:gens} by a set of generators of a dense subgroup.
\end{enumerate}
The approach \ref{rep:polys} is convenient, for example, for membership testing and for computing  
the dimension and the Lie algebra.
\ref{rep:gens} is useful for computing normalizers and centralizers.
Other ways of representing are available if some additional information (e.g., being connected or reductive) is known about the group.
We refer to~\cite[Section~3.13]{deGraaf} for a discussion.

In this paper, we will focus on~\ref{rep:polys}, the representation of linear algebraic groups by a system of defining equations.
It is known~\cite[Proposition~3]{Heintz} that an affine
variety can be defined by a system of polynomial equations of degree at most the degree of the variety. 
Thus, the degree of an algebraic group as an algebraic variety becomes a natural measure of complexity if~\ref{rep:polys} is chosen.
In addition to that, the degrees of an algebraic group and its orbits play an important role in constructive invariant theory~\cite{D99, D01} (see also~\cite{K87,B89} for bounds in the case of a reductive group).
The degree of a linear algebraic group can be arbitrarily large even in the case of $n = 1$ (see Example~\ref{ex:roots}).
However, as we show in this paper, every linear algebraic group can be ``approximated'' by a group of degree at most $(4n)^{3n^2}$ (see Section~\ref{sec:summary}).


\subsection{Hrushovski's algorithm}
Our main motivation comes from Hrushovski's algorithm for computing the Galois group of a linear differential equation~\cite{Hrushovski}.
A Galois group is associated to every linear differential equation and captures such properties of the solutions of the equation
as solvability by quadratures and algebraicity of relations among solutions (for details, see~\cite{SingerVanDerPut}).
Galois theory of differential equations has applications in integrable systems~\cite{MoralesRuiz} and number theory~\cite{Beukers}, among other areas.

In contrast with the Galois theory of polynomial equations, 
in which Galois groups are finite, 
the Galois group of a linear differential equation is a linear algebraic group and so it is usually infinite.
Moreover, it is known~\cite{Tretkoff,SingerMitschi,Hartmann} that every linear algebraic group over an algebraically closed field $C$ of characteristic zero can appear as the Galois group of a linear differential equation over $C(t)$.
Several algorithms were designed for computing the differential Galois group in special cases~\cite{Kovacic,CompointSinger,SingerUlmer},  computing invariants~\cite{vanHoeijWeil} and the Lie algebra~\cite{Barkatouetal} of the differential Galois group, and computing the differential Galois group approximately~\cite{VanDerHoeven}.

Let $C$ be an algebraically closed field of characteristic zero.
The first general algorithm for computing the Galois group of a linear differential equation over $C(t)$ due to Hrushovski~\cite{Hrushovski} appeared in 2002.
The algorithm was used, for example, to design algorithms for computing the Galois group of a linear differential equation with parameters in several cases~\cite{MOS14,MOS15}.
For the last decade, it has been a challenge to understand the complexity of Hrushovski's algorithm and make it practical, see~\cite{Feng,Sun,Rettstadt} for recent progress in this direction.
One of the key ingredients of the algorithm is the following fact, which is of independent interest to the effective and computational theory of algebraic groups.
There exists a function $\bm{d}(n)$ such that for every algebraic group $G \subset \GL_n(C)$ there exists an algebraic group $H \subset \GL_n(C)$ containing $G$ such that
\begin{enumerate}[label=\textbf{(H\arabic*)}]
  \item\label{h:approx} $H$ \emph{approximates} $G$ in the following sense: there is a set of characters of $H^\circ$ such that $G^\circ$ is equal to the intersection of their kernels ($X^\circ$ denotes the connected component of identity in $X$).
  \item\label{h:degree} $H$ can be defined by equations of degree \emph{at most} $\bm{d}(n)$. 
\end{enumerate}
Constructing such an approximation $H$ of the differential Galois group $G$ is the first step of Hrushovski's algorithm.
The bound on the degrees of defining equations allows one to search for defining equations of $H$ using undetermined coefficients.
With such an approximation $H$ at hand, the algorithm then proceeds to compute $G$ by using the algorithm by Compoint and Singer~\cite{CompointSinger}.

Hrushovski himself did not provide an explicit expression for $\bm{d}(n)$ but showed its existence~\cite[Corollary~3.7]{Hrushovski}.
He also conjectured that the overall complexity of the algorithm is at most double-exponential~\cite[Remark~4.4]{Hrushovski}.
Feng~\cite[Proposition~B.14]{Feng} found the first explicit formula for such a $\bm{d}(n)$ by presenting a function of quintuply exponential growth in $n$ that could be used as $\bm{d}(n)$ in~\ref{h:degree} (see also~\cite{Sun} for a related bound for $H$).


\subsection{Summary of the main results}\label{sec:summary}
In this paper, we show that every linear algebraic group $G \subset \GL_n(C)$ can be approximated in the sense of~\ref{h:approx} by a group of degree at most
\begin{equation}\label{eq:main_bound}
\begin{cases}[c]
   1 &\text{ for } n = 1,\\
   6 &\text{ for } n = 2,\\
   360 &\text{ for } n = 3,\\
   (4n)^{3n^2} &\text{ for } n > 3.
\end{cases}
\end{equation}
More precisely, we formalize~\ref{h:approx} by the notion of a \emph{toric envelope}.
We say that $H \subset \GL_n(C)$ is a \emph{toric envelope} of an algebraic group $G \subset \GL_n(C)$ if there exists a torus $T \subset \GL_n(C)$ such that $H$ can be written as a product $T \cdot G$ (as a product of abstract groups).
Theorems~\ref{thm:general_bound}, \ref{thm:n2}, and~\ref{thm:n3} state that every algebraic subgroup of $\GL_n(C)$ has a toric envelope of degree at most~\eqref{eq:main_bound}.
 
In particular, we show that one can take $\bm{d}(n) = (4n)^{3n^2}$ (see Section~\ref{sec:proto}).
This improves Feng's result dramatically.
Our bound is qualitatively optimal (see Remark~\ref{rem:optimal}) in the sense that any such bound is at least single-exponential.

\subsection{Outline of the approach}

We derive a general bound in Theorem~\ref{thm:general_bound} for a group $G \subset \GL_n(C)$ in the following steps.

    \paragraph{Divide.} By the Levi  decomposition, $G$ is a product of a reductive group and the unipotent radical.
    
    \paragraph{Conquer (reductive).} We find a toric envelope of bounded degree for the reductive group (Section~\ref{sub:reductive}). We reduce the problem to the case of a connected reductive group by deriving a bound (Lemma~\ref{lem:bound_components}) analogous to the Jordan bound for finite subgroups of $\GL_n(C)$~\cite{JordanBound}.
    We construct a toric envelope in the case of a connected group using the Lie correspondence and theory of reductive Lie algebras (Lemma~\ref{lem:connected_reductive}).
    We put everything together in Lemma~\ref{lem:reductive_bound}.
    
    \paragraph{Conquer (unipotent).} We derive a degree bound for any unipotent subgroup of $\GL_n(C)$ via representing it as the image of the exponential map (Section~\ref{sub:unipotent}).
    
    \paragraph{Combine.} We combine the obtained bounds to produce a toric envelope of $G$ (Section~\ref{sub:levi_bound}).

\noindent
The proof of Theorem~\ref{thm:n2} takes a different approach based on a classification of subgroups of $\SL_2(C)$ and computer-assisted computation of the degree bound for the hardest special cases (Section~\ref{sec:pf3}).
The proof of Theorem~\ref{thm:n3} refines the ideas from the proof of Theorem~\ref{thm:general_bound} in the case $n = 3$ (Section~\ref{sec:pf4}).

\subsection{Structure of the paper} 
The rest of the paper is organized as follows. Section~\ref{sec:prelim} contains the notions used to state the main results and illustrates them by examples.
Section~\ref{sec:main} contains the main results of the paper.
Section~\ref{sec:proto} describes the application of the main results to Hrushovski's algorithm for computing the differential Galois group of a linear differential equation.
Sections~\ref{sec:pf1}, \ref{sec:pf2}, \ref{sec:pf3}, and~\ref{sec:pf4} contain proofs of the main results.


\section{Preliminaries}
\label{sec:prelim}

Throughout the paper, $C$ denotes an algebraically closed field of characteristic zero.

\begin{definition}
  A \emph{torus} is a commutative connected algebraic subgroup $T \subset \GL_n(C)$ such that every element of $T$ is diagonalizable.
\end{definition}

\begin{definition}\hypertarget{def_envelope}{}
  Consider a linear algebraic group $G \subset \GL_n(C)$.
  We say that an algebraic group $H \subset \GL_n(C)$ is \emph{a toric envelope} of $G$ if there exists a torus $T \subset \GL_n(C)$ such that $H = T \cdot G$ (product as abstract groups). 
\end{definition}

\begin{remark}
  The product $T\cdot G$ is not necessarily semidirect.
  More precisely,
  $T$ does not necessarily normalize $G$ (Example~\ref{ex:dihedral}) and $G$ does not necessarily normalize $T$ (Example~\ref{ex:torus}).
\end{remark}

\begin{example}
  Every linear algebraic group $G \subset \GL_n(C)$ is a toric envelope of itself (with $T = \{e\}$).
\end{example}

\begin{definition}
  A algebraic subvariety $X \subset \GL_n(C)$ is said to be \emph{bounded by $d$}, where $d$ is a positive integer, if there exist polynomials $f_1, \ldots, f_M \in C[x_{11}, x_{12}, \ldots, x_{nn}]$ of degree at most $d$ such that 
  \[
  X = \GL_n(C) \cap \{ f_1 = f_2 = \ldots = f_M = 0 \}.
  \]
\end{definition}

The following notion of the degree of a variety (we specialize it to subvarieties of $\GL_n(C)$, see~\cite[Section~2]{Heintz} for a general treatment) is a generalization of the notion of degree of a polynomial.

\begin{definition}
  Let $X \subset \GL_n(C)$ be a subvariety such that all irreducible components of $X$ are of the same dimension $m$ (for example, this is the case if $X$ is a linear algebraic group).
  Then 
  \[
  \deg X := \max\left\{ |X \cap H| \;\colon\; H \text{ is a hyperplane of codimension } m \text{ such that } |X \cap H| \text{ is finite} \right\}.
  \]
\end{definition}

\noindent
By~\cite[Remark~2]{Heintz}, the degree of a hypersurface is equal to the degree of its defining polynomial.

\begin{proposition}[{{follows from~\cite[Proposition~3]{Heintz} }}]
  Let $X \subset \GL_n(C)$ be a subvariety of degree $D$.
  Then $X$ can be defined by equations of degree at most $D$.
\end{proposition}

The following examples show that the degree of a toric envelope of a group can be much smaller that the degree of the group itself.

\begin{example}\label{ex:roots}
  Let $N$  be a positive integer and $G \subset \GL_1(C)$ be the group of all $N$-th roots of unity.
  It is defined inside $\GL_{1}(C)$ by a single equation $x^N - 1 = 0$ of degree $N$, so it has degree $N$.
  The whole $\GL_1(C)$ is a toric envelope of $G$ with $T = \GL_1(C)$ and is of degree $1$.
\end{example}

\begin{example}\label{ex:torus}
  Consider 
  \[
    G = \left\{ \begin{pmatrix}
    a & b \\
    0 & a^{2018}
    \end{pmatrix} \;\mid\; a \in C^\ast, \; b \in C \right\}.
  \]
  One can show that the degree of $G$ is $2018$.
  Let $T$ be the group of all diagonal matrices.
  Then the group of all triangular matrices in $\GL_2(C)$ is a toric envelope of $G$ because it is equal to $T \cdot G$.
  This group is defined by a single linear equation, so it has degree $1$.
\end{example}

\begin{example}\label{ex:dihedral}
  Consider a dihedral group
  \[
  G = \left\{ \begin{pmatrix}
  \varepsilon^m & 0 \\
  0 & \varepsilon^{-m}
  \end{pmatrix} \;\mid\; m \in \mathbb{Z} \right\} \cup \left\{ \begin{pmatrix}
  0 & \varepsilon^m \\
  \varepsilon^{-m} & 0
  \end{pmatrix} \;\mid\; m \in \mathbb{Z} \right\},
  \]
  where $\varepsilon$ is a primitive $2018$-th root of unity.
  Since $G$ is a zero-dimensional variety consisting of $4036$ points, $\deg G = 4036$.
  Let $T$ be the group of all diagonal matrices.
  Then 
  \[
  T\cdot G = \left\{ \begin{pmatrix}
  a & 0 \\
  0 & b
  \end{pmatrix} \;\mid\; a, b \in C^\ast \right\} \cup \left\{ \begin{pmatrix}
  0 & a \\
  b & 0
  \end{pmatrix} \;\mid\; a, b \in C^\ast \right\}
  \]
  is a toric envelope of $G$. 
  Since $T\cdot G$ is a union of two two-dimensional spaces, $\deg (T \cdot G) = 2$.
\end{example}

\section{Main results}
\label{sec:main}
\begin{theorem}\label{thm:general_bound}
  Let $C$ be an algebraically closed field of characteristic zero.
  Let $G \subset \GL_n(C)$ be a linear algebraic group.
  Then there exists a \hyperlink{def_envelope}{toric envelope} $H$ of $G$ of degree at most $(4n)^{3n^2}$. 
  In particular, $H$ is bounded by this number.
\end{theorem}

\begin{theorem}\label{thm:n2}
   Let $C$ be an algebraically closed field of characteristic zero.
  Let $G \subset \GL_2(C)$ be a linear algebraic group.
  Then there exists a \hyperlink{def_envelope}{toric envelope} $H$ of $G$ such that $H$ is bounded by $6$.
\end{theorem}

\begin{theorem}\label{thm:n3}
  Let $C$ be an algebraically closed field of characteristic zero.
  Let $G \subset \GL_3(C)$ be a linear algebraic group.
  Then there exists a \hyperlink{def_envelope}{toric envelope} $H$ of $G$ of degree at most $360$. 
  In particular, $H$ is bounded by $360$.
\end{theorem}

\begin{remark}
  A sharper bound for Theorem~\ref{thm:general_bound} is given by~\eqref{eq:tighter_bound} in Section~\ref{sec:pf2}.
\end{remark}

\begin{remark}\label{rem:optimal}
  Let us show that the bound in Theorem~\ref{thm:general_bound} is qualitatively optimal by presenting a single-exponential lower bound.
  Fix a positive integer $n$.
  Let $D$ and $P$ be the group of all diagonal matrices and the group of all permutation matrices in this basis, respectively.
  Since $P$ normalizes $D$, their product $G := P\cdot D \subset \GL_n(C)$ is an algebraic group~\cite[\S 3 and Theorem~3 on p. 102]{VinbergOnishik}.
  One can show that, since $G^\circ$ is a maximal torus in $\GL_n(C)$, the only possible toric envelope of $G$ is $G$ itself.
  Since $P \cap D = \{ e \}$, the number of connected components of $G$ is equal to $|P| = n!$.
  Since $G^\circ = D$, every component has 
  degree $\deg D = 1$.
  Thus, we obtain a single-exponential lower bound
  \[
    \deg G = n! = n^{\mathrm{O}(n)}.
  \]
  The same example gives a single-exponential lower degree bound for a proto-Galois group (see Section~\ref{sec:proto}) as well.
\end{remark}


\section{Application to Hrushovski's algorithm}
\label{sec:proto}

Hrushovski's algorithm for computing the differential Galois group~\cite{Hrushovski} of a linear differential equation of order $n$ consists of the following three steps as outlined in~\cite[Section~1]{Feng}
\begin{enumerate}
  \item\label{step:proto} Computing a proto-Galois group of the differential Galois group of the equation using an a priori upper bound for the degrees of the defining equations.
  \item Compute the toric part using the algorithm by Compoint and Singer~\cite{CompointSinger}.
  \item Compute the finite part.
\end{enumerate}

In this section, we show (Lemma~\ref{lem:proto}) that every \hyperlink{def_envelope}{toric envelope} of an algebraic group $G \subset \GL_n(C)$ is a proto-Galois group of $G$.  It follows that the bounds from Theorems~\ref{thm:general_bound}, \ref{thm:n2}, and~\ref{thm:n3} can be used in the first step of Hrushovski's algorithm instead of the bound given in~\cite[Proposition~B.11]{Feng}.

Feng~\cite[Definition~1.1]{Feng} defined a proto-Galois group as follows.

\begin{definition}[proto-Galois group]\label{def:proto}
  Let $G \subset \GL_n(C)$ be an algebraic group.
  An algebraic group $H \subset \GL_n(C)$ is called a \emph{proto-Galois group} of $G$ if
  \[
    (H^\circ)^{t} \unlhd G^\circ \subset G \subset H,
  \]
  where $(H^\circ)^t$ denotes the intersection of the kernels of all characters of $H^\circ$.
\end{definition}

\begin{lemma}\label{lem:proto}
  If $H \subset \GL_n(C)$ is a \hyperlink{def_envelope}{toric envelope} of $G \subset \GL_n(C)$, then $H$ is a proto-Galois group of $G$.
\end{lemma}

\begin{proof}
  Since $H$ is a toric envelope of $G$, $G\subset H$, so inclusions $G^\circ \subset G \subset H$ from Definition~\ref{def:proto} hold.
  \cite[Lemma~2.1]{SingerModuli} implies that $(H^\circ)^t$ is exactly the subgroup of $H$ generated by all unipotent elements of $H^\circ$.
  By Lemma~\ref{lem:def_TE}, it coincides with the subgroup of $G^\circ$ generated by all unipotents in $G^\circ$, and such a subgroup is normal in $G^\circ$.
\end{proof}

\begin{corollary}
  Let $C$ be an algebraically closed field of characteristic zero.
  For every linear algebraic group $G \subset \GL_n(C)$
  \begin{itemize}
    \item there exists a proto-Galois group $H$  bounded by $(4n)^{3n^2}$;
    \item if $n = 2$, there exists a proto-Galois group $H$  bounded by $6$;
    \item if $n = 3$, there exists a proto-Galois group $H$  bounded by $360$.
  \end{itemize}
\end{corollary}


\section{Proof ingredients}
\label{sec:pf1}

\begin{notation}\label{not:main}
  In what follows we will use the following notation.
  \begin{itemize}[leftmargin=6mm, itemsep=0pt, topsep=1pt]
    \item By $C$ we denote an algebraically closed field of characteristic zero.
    \item We denote the set of all $n \times n$ (resp., $n \times m$) matrices over $C$ by $\Mat_n(C)$ (resp., $\Mat_{n, m}(C)$).
    \item We denote the subgroup of all scalar matrices in $\GL_n(C)$ by $Z_n \subset \GL_n(C)$.
    \item\hypertarget{N}{}   For a subset $X \subset \Mat_n(C)$, 
    we denote the normalizer and centralizer subgroups of $X$
      by $N(X)$ and $Z(X)$, respectively.
     \item For a subgroup $G \subset \GL_n(C)$, we denote the center by $C(G)$ and the connected component of the identity by $G^\circ$.
     \item For a Lie subalgebra $\mathfrak{u} \subset \mathfrak{gl}_n(C)$, we denote the normalizer and centralizer subalgebras 
      by $\mathfrak{n}(\mathfrak{u})$ and $\mathfrak{z}(\mathfrak{u})$, respectively.
        \item\hypertarget{Jordan}{} For a positive integer $n$, $J(n)$ is the minimal number such that every finite subgroup of $\GL_n(C)$ contains a normal abelian subgroup of index at most $J(n)$. 
  We will use Schur's bound~\cite[Theorem~36.14]{CurtisReiner} 
  \begin{equation}\label{eq:Jordan}
  J(n) \leqslant \big(\sqrt{8n} + 1\big)^{2n^2} - \big(\sqrt{8n} - 1\big)^{2n^2}
  \end{equation}
  \item\hypertarget{A}{} For a positive integer $n$, $A(n)$ is the maximal size of a finite abelian subgroup of $\GL_n(\mathbb{Z})$.
  Some known values are $A(1) = 2$, $A(2) = 6$ (see~\cite[p. 180]{Newman}), and $A(3) = 12$ (see~\cite[p. 170]{Tahara}), a general upper bound is given by Lemma~\ref{lem:An}.
  \end{itemize}
\end{notation}

\subsection{Auxiliary lemmas}

\begin{lemma}\label{lem:def_TE}
An algebraic group $H \subset \GL_n(C)$ is a \hyperlink{def_envelope}{toric envelope} of an algebraic group $G \subset \GL_n(C)$ if and only if
\begin{enumerate}[itemsep=-3pt, topsep=1pt]
  \item\label{property:unip} $H^\circ$ and $G^\circ$ have the same set of unipotents;
  \item\label{property:incl} $H = G\cdot H^\circ$.
\end{enumerate}
\end{lemma}

\begin{proof}
  Let $H$ be a toric envelope of $G$. 
  Then there exists torus $T \subset \GL_n(C)$ such that $H = T\cdot G$.
  Since $T$ is connected, $T \subset H^\circ$, so $H \supset G \cdot H^{\circ} \supset G \cdot T = H$, so~\eqref{property:incl} holds.
  Consider any unipotent element $A \in H^\circ$. 
  Since $H^\circ = T \cdot G^\circ$, then there are $B \in T$ and $C \in G^\circ$ such that $A = BC$.
  Since $T$ is a torus, $B$ is a semisimple element.
  Then $C = B^{-1}A$ is a Jordan-Chevalley decomposition of $C$.
  By~\cite[Theorem~6, p. 115]{VinbergOnishik}, $A, B \in G^\circ$.
  Thus, every unipotent element of $H^\circ$ belongs to $G^\circ$.
  Since also $G\subset H$, \eqref{property:unip} holds.
  
  Assume that properties~\eqref{property:unip} and~\eqref{property:incl} hold for $G$ and $H$.
  Let $H = H_0 \ltimes U$ be a Levi decomposition of $H$ (see~\cite[Theorem~4, p. 286]{VinbergOnishik}).
    By \cite[Lemma~10.10]{Wehrfritz}, $H_0$ can be written as a product $\Gamma H_0^\circ$ for some finite group $\Gamma \subset \GL_n(C)$.  
  \cite[Proposition, p. 181]{Borel} implies that $H_0^\circ$ can be written as $S T$, where $T := C(H_0^\circ)^\circ$ is a torus and $S := [H_0^\circ, H_0^\circ]$ is semisimple.
  Since $\Gamma$ normalizes $H_0$ and the center is a characteristic subgroup, $\Gamma$ normalizes $T$.
  Since $U$ and $S$ are generated by unipotents, $U, S \subset G^\circ$.
  Then
  \[
    H \supset T \cdot G = T \cdot G^\circ \cdot G \supset T \cdot S \cdot U \cdot G = H^\circ \cdot G = H, 
  \]
  so $H = T \cdot G$ and $H$ is a toric envelope of $G$.\qedhere
 \end{proof}

\begin{corollary}\label{cor:transitive}
  If $H_2 \subset \GL_n(C)$ is a \hyperlink{def_envelope}{toric envelope} of $H_1 \subset \GL_n(C)$ and $H_1$ is a  toric envelope of $H_0 \subset \GL_n(C)$, then $H_2$ is a toric envelope of $H_0$.
\end{corollary}

\begin{proof}
  Lemma~\ref{lem:def_TE} implies that $H_0^\circ$, $H_1^\circ$, and $H_2^\circ$ have the same set of unipotents.
  Since $H_1 \subset H_2$, we have $H_1^\circ \subset H_2^\circ$.
  Together with Lemma~\ref{lem:def_TE} this implies that $H_2 = H_2^\circ\cdot H_1 = H_2^\circ\cdot H_1^\circ\cdot H_0 = H_2^\circ\cdot  H_0$.
  Lemma~\ref{lem:def_TE} implies that $H_2$ is a toric envelope of $H_0$.
\end{proof}


\begin{corollary}\label{cor:TE_reductive}
  Any \hyperlink{def_envelope}{toric envelope} of a reductive group is again a reductive group. 
\end{corollary}

\begin{proof}
  Let $G$ be a reductive group and $H$ be a toric envelope of $G$.
  Assume that $H$ is not reductive.  Then it contains a nontrivial connected normal unipotent subgroup $U$.
  Since $G^\circ$ and $H^\circ$ have the same unipotents, $U \subset G^\circ$.
  This contradicts the reductivity of $G$.
\end{proof}


\begin{corollary}\label{cor:scalar}
  Let $G$ be an algebraic subgroup of $\GL_n(C)$.
  Then every \hyperlink{def_envelope}{toric envelope} of $GZ_n$ is a  toric envelope of $G$.
\end{corollary}


\begin{lemma}\label{lem:product}
  Let $G \subset \GL_n(C)$ be an algebraic group such that $Z_n \subset G$.
  Then $G = Z_n \cdot (G \cap \SL_n(C))$.
\end{lemma}

\begin{proof}
  Let $A \in G$. 
  Since $Z_n \subset G$, $\frac{1}{\sqrt[n]{\det A}} A \in G \cap \SL_n(C)$.
  Thus, $A \in Z_n \cdot (G \cap \SL_n(C))$.
\end{proof}

The following geometric lemma is a modification of~\cite[Lemma~3]{JeronimoSabia}.

\begin{lemma}\label{lem:degree_intersection}
  Let $X \subset \mathbb{A}^N$ be an algebraic variety of dimension $\leqslant d$ and degree $D$.
  Consider polynomials $f_1, \ldots, f_M \in C[\mathbb{A}^N]$ such that $\deg f_i \leqslant D_{1}$ for every $1 \leqslant i \leqslant M$.
  Then the sum of the degrees of the components of $Y := X \cap \{f_1 = \ldots = f_M = 0\}$ of dimension $\geqslant d'$ does not exceed $D \cdot D_1^{d - d'}$.
\end{lemma}

\begin{proof}
  We will prove the lemma by induction on $d - d'$.
  The base case is $d = d'$. 
  In this case, the set of components of $Y$ of dimension at least $d'$ is a subset of the set of components of $X$ of dimension $d$, and the sum of their degrees is at most $\deg X = D$.

  Let $d > d'$.
  Considering every component of $X$ separately, we may reduce to the case that $X$ is irreducible.
  If every $f_i$ vanishes on $X$, then $X = Y$ and the only component of $Y$ of dimension $\geqslant d'$ has degree $D$.
  Otherwise, assume that $f_1$ does not vanish everywhere on $X$.
  Then $\dim X \cap V(f_1) \leqslant d - 1$ and $\deg (X \cap V(f_1)) \leqslant DD_1$ by~\cite[Theorem 7.7, Chapter 1]{Hartshorne}.
  Applying the induction hypothesis to $X \cap V(f_1)$ and the same $d'$, we show that the sum of the degrees of the components of $Y$ of dimension at least $d'$ is at most  \[
  \deg (X \cap V(f_1)) D_1^{d - 1 - d'} \leqslant D D_1^{d - d'}.\qedhere
  \]
\end{proof}


\begin{corollary}\label{cor:intersect_with_normalizer}
  For every collection of algebraic subgroups  $G_0, G_1, \ldots, G_k \subset \GL_n(C)$
  \[
  \deg G \leqslant n^{\dim G_0 - \dim G} \deg G_0, \text{ where } G := G_0 \cap \hyperlink{N}{N}(G_1) \cap \ldots \cap \hyperlink{N}{N}(G_k).
  \]
\end{corollary}

\begin{proof}
  \cite[Lemma~B.4]{Feng} implies that $N(G_1) \cap \ldots \cap N(G_k)$ is defined by polynomials of degree at most $n$.
  Then the statement of the corollary follows from Lemma~\ref{lem:degree_intersection} with $D_1 = n$.
\end{proof}


\begin{lemma}\label{lem:An}
  $\hyperlink{A}{A}(n) \leqslant 2\cdot 3^{[n^2 / 4]}$, where $[x]$ means the integer part of $x$.
\end{lemma}

\begin{proof}
  Consider a finite abelian subgroup $A \subset \GL_n(\mathbb{Z})$.
  Let $A_0 := A \cap \SL_n(\mathbb{Z})$, then $|A| \leqslant 2 |A_0|$.
  Consider the homomorphism $\varphi\colon \SL_n(\mathbb{Z}) \to \SL_n(\mathbb{F}_3)$ defined by reducing modulo $3$.
  \cite[Theorem~\textrm{IX}.8]{Newman} implies that $|A_0| = |\varphi(A_0)|$.
  According to~\cite[Table~2]{Vdovin}, the size of an abelian subgroup  of $\SL_n(\mathbb{F}_3)$ does not exceed $3^{[n^2 / 4]}$.
\end{proof}


\subsection{Degree bound for unipotent groups}\label{sub:unipotent}

   \begin{lemma}\label{lem:unipotent_bound}
    	Let $U \subset \GL_n(C)$ be a connected unipotent group.
        Then 
        \[
        \deg U \leqslant \prod\limits_{k = 1}^{n - 1} k!.
        \]
    \end{lemma}
    
    \begin{proof}
      By Engel's theorem~\cite[Corollary~1, p. 125]{VinbergOnishik}, there exists a basis such that $\Lie U$ is contained in a subspace $\mathcal{T} \subset \Mat_n(C)$ of strictly upper triangular matrices.  From now on, we fix such a basis.  
      By~\cite[Theorem~7, p. 126]{VinbergOnishik}, $U=\varphi(\Lie U)$, where $\varphi$ is the exponential map.
      Since every matrix in $\mathcal{T}$ is nilpotent of index at most $n$, $\varphi$ is defined everywhere on $\mathcal{T}$ by the following formula
    \[
      \varphi(X) = I_n + X + \frac{X^2}{2!} + \ldots + \frac{X^{n - 1}}{(n - 1)!} \text{ for } X \in \mathcal{T}.
    \]
    Consider the affine
    variety
    \[
    W := \{ (X, Y) \in \mathcal{T}\times \GL_n(C) \;|\; Y = \varphi(X) \;\&\; X \in \Lie U \}.
    \]
    Since the projection of $W$ to
    $\GL_n(C)$ is equal to $\varphi(\Lie U) = U$, $\deg U \leqslant \deg W$ by~\cite[Lemma~2]{Heintz}.
    The condition $X \in \Lie U$ is defined by linear equations.
    A direct computation shows that 
    \[
    \deg (\varphi(X))_{i, j} \leqslant 
    \begin{cases}
      -\infty, \text{ if } i > j,\\
      j - i, \text{ otherwise},
    \end{cases}
    \]
    where 
    $(\varphi(X))_{i, j}$ 
    denotes the $(i, j)$-th entry of the matrix $\varphi(X)$ whose entries are polynomials in the entries of $X$.
    The condition $Y = \varphi(X)$ is defined by $\frac{n(n + 1)}{2}$ linear equations, $n - 2$ quadratic equations, $n - 3$ equations of degree $3$, $\ldots$, one equation of degree $n - 1$.
    Thus, Bezout's theorem~\cite[Theorem~7.7, Chapter~1]{Hartshorne} implies that $\deg W \leqslant 2^{n - 2} 3^{n - 3} \ldots (n - 2)^2(n - 1) = \prod\limits_{k = 1}^{n - 1} k!$.
    \end{proof}

\subsection{Degree bound for reductive groups}\label{sub:reductive}
All statements in this section will be about a reductive group $G \subset \GL_n(C)$ such that $G \subset \hyperlink{N}{N}(F)$, where $F \subset \GL_n(C)$ is some connected group.
In our proofs, $G$ and $F$ will be the reductive and unipotent parts of a Levi decomposition of an arbitrary linear algebraic group, respectively.

\begin{lemma}\label{lem:bound_components}
  Let $G \subset \GL_n(C)$ be a reductive algebraic group such that $G \subset \hyperlink{N}{N}(F)$ for some connected  algebraic group $F \subset \GL_n(C)$.
  Then there is a \hyperlink{def_envelope}{toric envelope} $H \subset \hyperlink{N}{N}(F)$ of $G$  such that
  \[
  [H : H^\circ] \leqslant \hyperlink{Jordan}{J}(n) \hyperlink{A}{A}(n - 1) n^{n - 1}.
  \]
\end{lemma}

\begin{proof}
  Using Corollary~\ref{cor:scalar}, we replace $G$ with $G Z_n$, so in what follows, we assume that $Z_n \subset G$.

  By \cite[Lemma~10.10]{Wehrfritz}, $G$ can be written as a product $\Gamma G^\circ$ for some finite group $\Gamma \subset \GL_n(C)$.  
  \cite[Proposition, p. 181]{Borel} implies that $G^\circ$ can be written as $S T$, where $T := C(G^\circ)^\circ$ is a torus and $S := [G^\circ, G^\circ]$ is semisimple.
  Since centers and connected components of the identity 
  are preserved by any automorphism of a group, and conjugation by an element of $\Gamma$ induces an automorphism of $G^{\circ}$,
  $\Gamma$ normalizes $T$.
  
  By the definition of $J(n)$ (see Notation~\ref{not:main}), there exists a normal abelian subgroup $\Gamma_{\mathrm{ab}} \subset \Gamma$ of index at most $J(n)$.
  Since $Z_n \subset G$, $T$ contains $Z_n$.
  Then Lemma~\ref{lem:product} implies that
  \[
  T = Z_n \cdot (T \cap \SL_n(C)).
  \]
  For every algebraic group $F$, by $\operatorname{AlgAut}(F)$ we will denote the group of algebraic automorphisms of $F$.
  The action of $\Gamma_{ab}$ on $T$ by conjugation defines a group homomorphism $\varphi\colon \Gamma_{ab} \to \operatorname{AlgAut} (T \cap \SL_n(C))$.
  Since $T \cap \SL_n(C) \cong (C^\ast)^d$  for some $d \leqslant n - 1$ (see~\cite[Problem~10, p. 114]{VinbergOnishik}), 
  \[
  \operatorname{AlgAut} (T \cap \SL_n(C)) \cong \GL_d(\mathbb{Z}).
  \]
  Let $\Gamma_0 := \Ker \varphi$.
  Since $\Gamma_0 = \Gamma_{\mathrm{ab}} \cap Z(T)$ and both $\Gamma_{\mathrm{ab}}$ and $T$ are normalized by $\Gamma$, $\Gamma_0$ is a normal subgroup in $\Gamma$.
  
  We set $H_0$ to be the intersection of all the maximal tori in $\GL_n(C)$ containing $\Gamma_0 \cdot T$.
  Since $\Gamma_0 \cdot T$ is a quasitorus, it is diagonalizable (see~\cite[Theorem~3, p. 113]{VinbergOnishik}), so there is at least one maximal torus containing $\Gamma_0 \cdot T$.
  Thus, $H_0$ is a torus.
  Since $\Gamma_0 \cdot T$ is normalized by $\Gamma$, $H_0$ is also normalized by $\Gamma$.
  We set $H_1 = H_0 \cap N(S) \cap N(F)$ and
  \begin{equation}\label{eq:Hdef_red}
    H := T_0 \cdot G, \text{ where } T_0 := H_1^\circ.
  \end{equation}
  
  The lemma follows from the following two claims.
  
  \paragraph{Claim~1: $H$ is a group.}
  Since $T \subset T_0$ and $\Gamma$ normalizes $T_0$, we have
  \begin{equation}\label{eq:H_group}
    H = T_0 \cdot G = T_0 \cdot \Gamma \cdot T \cdot S = \Gamma \cdot (T_0 \cdot S).
  \end{equation}
  The latter is a group, because $T_0$ normalizes $S$ and $\Gamma$ normalizes $T_0$ and $S$.
    
  \paragraph{Claim~2: $[H : H^\circ] \leqslant J(n) A(n - 1) n^{n - 1}$.}
  From~\eqref{eq:H_group} we have $H = (\Gamma \cdot T_0) \cdot (T_0\cdot S)$.
  Since $T_0\cdot S$ is connected, $H$ has at most as many connected components as $\Gamma \cdot T_0$.
  Since $T_0 = H_1^\circ$, the latter is bounded by the number of connected components of $\Gamma\cdot H_1$.
  We have
  \begin{equation}\label{eq:gamma_h1}
  \deg (\Gamma \cdot H_1) \leqslant [\Gamma : \Gamma_0] \cdot \deg H_1 = [\Gamma : \Gamma_{\mathrm{ab}}] \cdot [\Gamma_\mathrm{ab} : \Gamma_{0}] \cdot \deg H_1.
  \end{equation}
  We have already shown that $[\Gamma : \Gamma_{\mathrm{ab}}] \leqslant J(n)$.
  The index $[\Gamma_\mathrm{ab} : \Gamma_{0}] = |\varphi(\Gamma_{\mathrm{ab}})|$ does not exceed the maximal size of a finite abelian subgroup of $\GL_d(\mathbb{Z})$.
  Since $d \leqslant n - 1$, this number is at most $A(n - 1)$.

  Since $H_0$ is defined by linear polynomials, $\deg H_0 = 1$.
  Since $H_0$ is a torus, $\dim H_0 \leqslant n$.
  Since $\dim \left(H_0 \cap N(S) \cap N(F) \right) \geqslant \dim Z_n = 1$, Corollary~\ref{cor:intersect_with_normalizer} implies that
  \begin{equation}\label{eq:bound_components}
  \deg H_1 = \deg \left(H_0 \cap N(S) \cap N(F)\right) \leqslant n^{n - 1}.
  \end{equation}
  Thus, $H$ has at most $[\Gamma \colon \Gamma_0] \cdot \deg H_1 \leqslant J(n) A(n - 1) n^{n - 1}$ connected components.\qedhere
  
\end{proof}

\begin{corollary}\label{cor:reductive_torus}
  In the notation of Lemma~\ref{lem:bound_components},
  if $G^\circ$ is a torus, then 
    \[
    \deg H \leqslant \hyperlink{Jordan}{J}(n) \hyperlink{A}{A}(n - 1) n^{n - 1}.
    \]
\end{corollary}

\begin{proof}
  In this case, $S$ from the proof of Lemma~\ref{lem:bound_components} is trivial.
  Since $T \subset T_0$, $H = \Gamma \cdot T_0$.
  Then $\deg H \leqslant (\deg \Gamma\cdot H_1)$.
  The latter is bounded by $J(n)A(n - 1)n^{n - 1}$ due to~\eqref{eq:gamma_h1} and~\eqref{eq:bound_components}.
\end{proof}

\begin{lemma}\label{lem:connected_reductive}
  Let $G \subset \GL_n(C)$ be a connected reductive group such that $G \subset \hyperlink{N}{N}(F)$ for some connected group $F \subset \GL_n(C)$.
  Then there exists a \hyperlink{def_envelope}{toric envelope} $H \subset \hyperlink{N}{N}(F)$ of $G$ such that
  \[
    \deg H \leqslant n^{n^2 - \dim G} \quad\text{ and }\quad \hyperlink{N}{N}(G) \cap \hyperlink{N}{N}(F) \subset \hyperlink{N}{N}(H).
  \]
\end{lemma}

\begin{proof}
  Using Corollary~\ref{cor:scalar} we may replace $G$ with $G Z_n$, so we will assume that $Z_n \subset G$.
   We set
   \[
   H := \left( Z(G) \cap Z(Z(G)) \cap N(F)\right)^\circ \cdot G.
   \]
   The lemma follows from the following three claims
   
   \paragraph{Claim~1: $H$ is a  toric envelope of $G$.}
   Since $Z(G)$ normalizes $G$, $H$ is a group. 
   We will show that the connected component of identity of $Z(G) \cap Z(Z(G))$ is a torus. 
   Then the connected component of the identity of $Z(G) \cap Z(Z(G)) \cap N(F)$ will also be a torus.
   
   Since $G$ is reductive, its representation in $C^n$ is completely reducible (see~\cite[Theorem~4.3, p. 117]{Hochschild}).
   Let $C^n = V_1 \oplus V_2 \oplus \ldots \oplus V_\ell$ be a decomposition of $C^n$ into isotypic components.
   Each $V_i$ can be written as $W_i \otimes C^{n_i}$, where $W_i$ is the corresponding irreducible representation of $G$ and $C^{n_i}$ is a trivial representation.
   Let $d_i := \dim W_i$ for $1 \leqslant i \leqslant \ell$.
   Then Schur's lemma implies that
   \[
   Z(G) = \bigoplus_{i = 1}^\ell \left( C^{\ast} I_{d_i} \otimes \GL_{n_i}(C) \right), \text{ where } I_{d_i} \text{ is a } d_i \times d_i \text{ identity matrix.}
   \]
   Since $Z(G) \cap Z(Z(G))$ is the center of $Z(G)$, we have
   
   \[
   Z(G) \cap Z(Z(G)) = \bigoplus_{i = 1}^\ell \left( C^{\ast} I_{d_i} \otimes C^\ast I_{n_i} \right) = \bigoplus_{i = 1}^\ell C^{\ast} I_{n_id_i}.
   \]
   Thus, $Z(G) \cap Z(Z(G))$ is a torus.  So the claim is proved.
   
   \paragraph{Claim~2: $\deg H \leqslant n^{n^2 - \dim G}$.}
   \cite[Proposition, p. 181]{Borel} implies that $G$ can be written as $S T$, where $T := C(G)^\circ$ is a torus and $S := [G, G]$ is semisimple.
   Consider $\widehat{H} := N(S) \cap Z(Z(G)) \cap N(F)$.
   Then $H \subset \widehat{H}$.
   We will show that $H = \widehat{H}^\circ$.

  Let $\mathfrak{g} := \Lie G$ and $\mathfrak{s} := \Lie S$.
  Consider an element $a \in \mathfrak{n}(\mathfrak{s})$.
  The map $\mathfrak{s} \to \mathfrak{s}$ defined by $g \mapsto [a, g]$ satisfies the requirements of Whitehead's lemma~\cite[Lemma~3, p. 77]{Jacobson}.
  
  Hence there exists $h \in \mathfrak{s}$ such that $[h, g] = [a, g]$ for every $g \in \mathfrak{s}$, so $a$ can be written as $a = h + (a - h)$, where $a - h \in \mathfrak{z}(\mathfrak{s})$.
  Since $\mathfrak{s}$ is semisimple, $\mathfrak{s} \cap \mathfrak{z}(\mathfrak{g}) = 0$, so 
  \begin{equation}\label{eq:normalizer}
  \mathfrak{n}(\mathfrak{s}) = \mathfrak{s} \oplus \mathfrak{z}(\mathfrak{s}).
  \end{equation}
   Decomposition~\eqref{eq:normalizer} implies that
   \begin{equation}\label{eq:whitehead_cor}
   N(S)^\circ = S \cdot Z(S)^\circ.
   \end{equation}
   We can write $\widehat{H}^\circ$ as
   \[
   \widehat{H}^\circ = \left( N(S) \cap Z(Z(G)) \cap N(F) \right)^\circ.
   \]
   Using~\eqref{eq:whitehead_cor}, we obtain
   \[
   \widehat{H}^\circ = \left( (S \cdot Z(S)^\circ) \cap Z(Z(G)) \cap N(F) \right)^\circ.
   \]
   Using consequently the inclusions $S \subset Z(Z(G)) \cap N(F)$ and $Z(Z(G)) \subset Z(T)$, we can further write
   \[
   \widehat{H}^\circ = \left( Z(S)^\circ \cap Z(Z(G)) \cap N(F) \right)^\circ \cdot S = \left( (Z(S) \cap Z(T)) \cap Z(Z(G)) \cap N(F) \right)^\circ \cdot S
   \]
   Since $Z(G) = Z(S) \cap Z(T)$, the latter is equal to $\left( Z(G) \cap Z(Z(G)) \cap N(F) \right)^\circ \cdot S$.
   Using $T \subset Z(G) \cap Z(Z(G)) \cap N(F)$, we conlcude that
   \[
     \widehat{H}^\circ = \left( Z(G) \cap Z(Z(G)) \cap N(F) \right)^\circ \cdot S = \left( Z(G) \cap Z(Z(G)) \cap N(F) \right)^\circ \cdot G = H.
   \]
   Thus, $\deg H \leqslant \deg \widehat{H}$.
   Since any centralizer is defined by linear equations, $\deg Z(Z(G)) = 1$.
  By Corollary~\ref{cor:intersect_with_normalizer} 
  \[
   \deg \widehat{H} \leqslant n^{\dim Z(Z(G)) - \dim \widehat{H}} \deg Z(Z(G)) \leqslant n^{n^2 - \dim G}.
   \]
   
   \paragraph{Claim~3. $N(G) \cap N(F) \subset N(H)$.}
   Consider $A \in N(G) \cap N(F)$.
   Since $A$ normalizes $G$, it normalizes $Z(G)$.
   Likewise, $A$ normalizes $Z(Z(G))$.
   Since $A$ also normalizes $N(F)$, we have $A \in N(H)$.
\end{proof}


\begin{lemma}\label{lem:reductive_bound}
 Let $G \subset \GL_n(C)$ be a reductive subgroup such that $G \subset \hyperlink{N}{N}(F)$ for some connected group $F \subset \GL_n(C)$.
 Then there exists a \hyperlink{def_envelope}{toric envelope} $H \subset \hyperlink{N}{N}(F)$ of $G$ such that
 \[
   \deg H \leqslant \hyperlink{Jordan}{J}(n) \hyperlink{A}{A}(n - 1) n^{n^2 + n - 5}.
 \]
\end{lemma}

\begin{proof}
  Using Corollary~\ref{cor:scalar}, we may replace $G$ with $G Z_n$,  so we will assume that $Z_n \subset G$.
  In the case that $G^\circ$ is a torus, the lemma follows from Corollary~\ref{cor:reductive_torus}.
  Otherwise, $\dim G \geqslant \dim Z_n + \dim \SL_2(C) = 4$.

  Since being a toric envelope is a transitive relation (see Corollary~\ref{cor:transitive}), 
  applying Lemma~\ref{lem:bound_components}, 
  we will further assume that $[G : G^\circ] \leqslant J(n) A(n - 1) n^{n - 1}$.
  \cite[Lemma~10.10]{Wehrfritz} implies that $G = \Gamma G^{\circ}$ for some finite group $\Gamma$.
  Lemma~\ref{lem:connected_reductive} implies that there exists a toric envelope $H_0 \subset N(F)$ of $G^\circ$ such that $N(G^\circ) \cap N(F) \subset N(H_0)$ and $\deg H_0 \leqslant n^{n^2 - 4}$.
  Let $H := \Gamma H_0$. 
  Since $\Gamma \subset N(G^\circ) \cap N(F) \subset N(H_0)$, $H$ is an algebraic group.
  Since $G^\circ \subset H_0$, $G \subset H$.
  Since $H^\circ = H_0$, all unipotent elements of $H^\circ$ belong to $G^\circ$.
  Since also $H^\circ G \supset H^\circ \Gamma = H$, Lemma~\ref{lem:def_TE} implies that $H$ is a toric envelope of $G$.
  
  Since $H = H_0 G$ where $H_0 = H^\circ$, $[H : H^\circ] \leqslant [G : G^\circ]$.
  Then
  \[
  \deg H = [H : H^\circ] \cdot \deg H^\circ \leqslant J(n) A(n - 1) n^{n - 1} \cdot n^{n^2 - 4} = J(n) A(n - 1) n^{n^2 + n - 5}. \qedhere
  \]
\end{proof}

\subsection{Degree bound for product}\label{sub:levi_bound}

\begin{lemma}\label{lem:levi_bound}
  Let $G = G_0 \cdot U \subset \GL_n(C)$, where $U$ is a connected unipotent group, $G_0$ is a reductive group, and $G \subset \hyperlink{N}{N}(U)$.
  Let $\deg G_0 = D_1$, $\deg U = D_2$.
  Then $\deg G \leqslant D_1D_2 2^{n(n - 1) / 2}$.
\end{lemma}

\begin{proof}
  The ambient space $C^n$ carries a filtration by subspaces
  \[
    V_i := \{ v \in C^n \;|\; \forall A \in U \;\; (A - I_n)^i v = 0 \}.
  \]
  There exists $s < n$ such that $V_1 \subsetneq V_2 \subsetneq \ldots \subsetneq V_s = C^n$.
  Since $G_0$ normalizes $U$, $V_i$ is invariant with respect to $G_0$ for every $i \geqslant 0$.
  Since $G_0$ is reductive, there exists a decomposition $V_{i} = V_{i - 1} \oplus W_i$ into a direct sum of $G_0$-representations for every $1 \leqslant i \leqslant s$.
  Thus, there is a decomposition $C^n = W_1 \oplus \ldots \oplus W_s$ into a direct sum of $G_0$-invariant subspaces.
  Let $n_i := \dim W_i$ for $1 \leqslant i \leqslant s$.
  We fix a basis of $C^n$ that is a union of bases of $W_1, \ldots, W_s$.
 In this basis, every element of $G_0$ is of the form
 \begin{equation}\label{eq:diag_mat}
 \begin{pmatrix}
   X_1 & 0 &  \ldots & 0 \\
   0   & X_2 &\ldots & 0 \\
   \vdots & \vdots & \ddots & \vdots \\
   0 & 0 & \ldots & X_s
 \end{pmatrix},
 \text{ where } X_i \in \Mat_{n_i}(C) \text{ for every } 1\leqslant i \leqslant s.
 \end{equation}
 And every element of $U$ is of the form
 \begin{equation}\label{eq:triang_mat}
 \begin{pmatrix}
   I_{n_1} & Y_{12} &  \ldots & Y_{1n} \\
   0   & I_{n_2} &\ldots & Y_{2n} \\
   \vdots & \vdots & \ddots & \vdots \\
   0 & 0 & \ldots & I_{n_s}
 \end{pmatrix},
 \text{ where } Y_{ij} \in \Mat_{n_i, n_j}(C) \text{ for every } 1\leqslant i < j \leqslant s.
 \end{equation}
 We denote the spaces of all the invertible matrices of the form~\eqref{eq:diag_mat} and~\eqref{eq:triang_mat} by $\mathcal{D}$ and $\mathcal{T}$, respectively.
  Consider the following variety
  \[
  P := \{ (X, Y, Z) \in \mathcal{D} \times \mathcal{T} \times \GL_n(C) \;|\; X \in G_0,\; Y\in U, \;XY = Z \} \subset \mathcal{D} \times \mathcal{T} \times \GL_n(C).
  \]
  Let $\pi\colon \mathcal{D} \times \mathcal{T} \times \GL_n(C) \to \GL_n(C)$ be the projection onto the last coordinate.
  Then $G = \pi(P)$, so $\deg G \leqslant \deg P$.
  Consider $P$ as an intersection of the variety $G_0 \times U \times \GL_n(C)$ of degree $D_1D_2$ with the variety defined by the $n^2$ equations $XY = Z$.
  Since the product $XY$ is of the form
  \[
  \begin{pmatrix}
    X_1 & X_1 Y_{12} & \ldots & X_1 Y_{1n} \\
    0 & X_2 & \ldots & X_2 Y_{2n} \\
    \vdots & \vdots & \ddots  & \vdots \\
    0 & 0 & \ldots & X_n
  \end{pmatrix},
  \]
  out of $n^2$ entries of $XY - Z$ there are 
  \begin{equation}\label{eq:prod_precise}
  \frac{n(n - 1)}{2} - \frac{n_1(n_1 - 1)}{2} - \ldots - \frac{n_s(n_s - 1)}{2} \leqslant \frac{n(n - 1)}{2}
  \end{equation}
  quadratic polynomials and the rest are linear.
  Thus, $\deg G \leqslant \deg P \leqslant D_1D_2 2^{n(n - 1) / 2}$.
\end{proof}


\section{Proof of Theorem~\ref{thm:general_bound}}
\label{sec:pf2}

\begin{proof}[Proof of Theorem~\ref{thm:general_bound}]
By~\cite[Theorem~4, p. 286]{VinbergOnishik}, $G$ can be written as a semidirect product $G_0 \ltimes U$, where $U$ is the unipotent radical of $G$ and $G_0$ is a reductive subgroup of $G$.

We apply Lemma~\ref{lem:reductive_bound} with $G = G_0$ and $F = U$ and obtain a toric envelope $H_s \subset N(U)$ of $G_0$.
Let $H_s = T\cdot G_0$, where $T$ is a torus.
We set $H := H_s\cdot U$.
Since $H_s \subset N(U)$, $H$ is an algebraic group.
Since $H = T\cdot G_0\cdot U = T\cdot G$, $H$ is a toric envelope of $G$.
Corollary~\ref{cor:TE_reductive} implies that $H_s$ is reductive.
Then Lemma~\ref{lem:levi_bound} implies that
\[
  \deg H \leqslant 2^{n(n - 1) / 2} \deg H_s \deg U.
\]
Using bounds for $H_s$ and $U$ from Lemmas~\ref{lem:reductive_bound} and~\ref{lem:unipotent_bound}, respectively, we obtain
\begin{equation}\label{eq:tighter_bound}
  \deg H \leqslant \hyperlink{Jordan}{J}(n) \hyperlink{A}{A}(n - 1) n^{n^2 + n - 5} 2^{(n - 1)n / 2} \prod\limits_{k = 1}^{n - 1} k!. 
\end{equation}
Using $\sqrt{8n} + 1 < \sqrt{16n}$ and~\eqref{eq:Jordan}, we derive $J(n) \leqslant 4^{2n^2} n^{n^2}$.
Using Lemma~\ref{lem:An} and $2\cdot 3^{(n - 1)^2 / 4} \leqslant 4^{n^2}$, we derive $A(n - 1) \leqslant 4^{n^2}$.
Using $n! \leqslant (\frac{n + 1}{2})^n$, we derive
\[
  \prod\limits_{k = 1}^{n - 1} k! \leqslant \left( \frac{n}{2} \right)^{n(n - 1) / 2}.
\]
Substituting all these bounds to~\eqref{eq:tighter_bound}, we obtain $\deg H \leqslant (4n)^{3n^2}.$
\end{proof}

\section{Proof of Theorem~\ref{thm:n2}}
\label{sec:pf3}
\begin{proof}
Using Corollary~\ref{cor:scalar}, from now on we assume that $G$ contains $Z_2$, the group of all scalar matrices in $\GL_2(C)$.
Lemma~\ref{lem:product} implies that $G = Z_2 G_{\SL}$, where $G_{\SL} := G\cap \SL_2(C)$.
According to \cite[p.7]{Kovacic}, there are only four options for $G_{\SL}$.
\begin{enumerate}[label=(\alph*), itemsep=0pt, topsep=1pt]
  \item\label{case:triangularizable}
  $G_{SL}$ is triangularizable but not diagonalizable.
  \item\label{case:diag}
  $G_{SL}$ is conjugate to a subgroup of
  \[D = \bigg{\{} \begin{pmatrix} a & 0 \\ 0 & d \end{pmatrix} \mid a, d \in C, ad \neq 0 \bigg{\}} \cup \bigg{\{}\begin{pmatrix} 0 & b \\ c & 0 \end{pmatrix} \mid b, c \in C, bc \neq 0 \bigg{\}}\]
  \item\label{case:finite}
  $G_{SL}$ is finite and neither of the previous two cases hold.
  \item\label{case:sl2}
  $G_{SL} = SL_2$
\end{enumerate}
We examine each of these cases individually below.

\paragraph{Case~\ref{case:triangularizable}: $G_{\SL}$ is triangularizable but not diagonalizable.} 
We fix a basis in which $G_{\SL}$ can be represented by upper-triangular matrices.
In this basis, $G$ is also represented by upper-triangular matrices.
Consider any non diagonalizable matrix in $G$
\begin{equation}\label{eq:Jordan_dec}
A = \begin{pmatrix} a & b \\ 0 & c  \end{pmatrix} = \begin{pmatrix} a & 0 \\ 0 & c  \end{pmatrix}\begin{pmatrix} 1 & a^{-1}b \\ 0 & 1  \end{pmatrix}.
\end{equation}
The last expression in~\eqref{eq:Jordan_dec} is the Jordan decomposition of $A$.
By \cite[Theorem~6, p. 115]{VinbergOnishik}, $\begin{pmatrix} 1 & a^{-1}b \\ 0 & 1  \end{pmatrix}$ belongs to $G$.
Since $a^{-1}b \neq 0$, the powers of this matrix generate a Zariski dense subgroup in the group $U$ of unipotent upper-triangular matrices.
Thus, $U \subset G$.
Consider the group $B$ of all invertible upper-triangular matrices.
Since $B$ is connected and the set of unipotent elements of $B$ is $U$,
Lemma~\ref{lem:def_TE} implies that $B$ is a toric envelope of $G$.
$B$ is defined by linear equations, so it is bounded by $1$.
Moreover, $B$ is a variety of dimension $3$ and degree $1$.

\paragraph{Case~\ref{case:diag}: $G_{\SL}$ is a subgroup of $D$.}
Then $G$ is also contained in $D$.
Since $D^\circ$ is the group of diagonal matrices, the only unipotent element in $D^\circ$ is the identity matrix.
Since $D$ has two connected components, Lemma~\ref{lem:def_TE} implies that either $D$ is a toric envelope of $G$ or, if $G$ is diagonalizable, $D^\circ$ is a toric envelope of $G$.
Since $D$ is defined inside $\GL_2(C)$ by quadratic equations $x_{11}x_{12} = x_{21}x_{22} = 0$, it is bounded by $2$. 
$D^\circ$ is defined by linear equations and bounded by $1$.
Moreover, $D$ is a variety of degree $2$ and dimension $2$ and $D^\circ$ is of degree $1$ and dimension $2$.

\paragraph{Case~\ref{case:finite}: $G_{\SL}$ is finite and neither of the two previous cases hold.}
We will use a classification of finite subgroups of $\SL_2(\mathbb{C})$ given in~\cite[\S\S 101-103]{MBD} (for a more modern treatment, see~\cite[Chapter~2]{Serrano}).
Based on this classification, $G_{\SL}$ must be one of the following five types:

\begin{enumerate}[itemsep=-1pt,topsep=1pt]
\item\label{case:cycl}
Cyclic Groups~\cite[\S 101, case~A]{MBD}
\item\label{case:dihedr}
Binary Dihedral Group~\cite[\S 101, case~B]{MBD}
\item\label{case:tetra}
Binary Tetrahedral Group~\cite[\S 102, case~C]{MBD}
\item\label{case:octa}
Binary Octahedral Group~\cite[\S 102, case~D]{MBD}
\item\label{case:icosa}
Binary Icosahedral Group~\cite[\S 103, case~E]{MBD}
\end{enumerate}
In 
cases~\eqref{case:cycl} and~\eqref{case:dihedr}, $G_{\SL}$ is conjugate to a subgroup of $D$ (see~\cite[\S 101, case~A and~B]{MBD}), so these cases are already considered in Case~\ref{case:diag}.
For cases~\eqref{case:tetra}-\eqref{case:icosa}, we take $G = Z_2G_{\SL}$ to be a toric envelope of itself.
Then we find a $d$ such that $G$ is bounded by $d$ using Algorithm~\ref{alg:degree} (for a {\sc Maple} code, see~{\url{https://github.com/pogudingleb/ToricEnvelopes.git}}). 

\begin{algorithm}[H]
\caption{Finding a degree bound for a toric envelope of a finite subgroup of $\SL_2(C)$}
\label{alg:degree}
\begin{description}[itemsep=0pt, topsep=1pt]
  \item[Input] Generators of a finite subgroup $G_{\SL} \subset \SL_2(C)$

 \item[Output] Positive integer $d$ such that $G := Z_2 G_{\SL}$ is bounded by $d$
 
\begin{enumerate}[label=(\arabic*),itemsep=0pt, topsep=1pt]
  \item Compute the list of all elements in $G_{\SL}$
  \item Compute generators of the vanishing ideal $I$ of $G := Z_2 G_{\SL}$
   \item Compute the reduced Gr\"obner basis $B$ of $I$ with respect to graded lexicographic ordering
   \item Check all $d$ from $1$ to $\max\limits_{p \in B} \deg p$ and find the minimum $d$
   such that
   \[
   \sqrt{ (p \in B | \deg p \leqslant d) } = I
   \]
   \item Return $d$ found in the previous step
\end{enumerate}
\end{description}
\end{algorithm}
Algorithm~\ref{alg:degree} outputs bounds 3, 4, and 6 for cases~\eqref{case:tetra}, \eqref{case:octa}, and~\eqref{case:icosa}, respectively.
The dimension of $G$ is one.
Since each of the groups~\eqref{case:tetra}  \eqref{case:octa}, and~\eqref{case:icosa} contains~$\begin{pmatrix} -1 & 0 \\ 0 & -1 \end{pmatrix}$, the degree of $G$ is at most half of the maximum of the cardinalities of these groups. 
Thus, $\deg G \leqslant 60$.

\paragraph{Case~\ref{case:sl2}: $G_{\SL} = \SL_2(C)$.} Then $G = Z_2 \SL_2(C) = \GL_2(C)$ is bounded by $0$.
In this case, $G$ is a variety of dimension $4$ and degree $1$.

Collecting together the results for cases~\ref{case:triangularizable}-\ref{case:sl2}, we conclude that every algebraic subgroups $G \subset \SL_2(C)$ has a toric envelope bounded by six.
\end{proof}

From the proof of Theorem~\ref{thm:n2}, we can extract additional information about possible toric envelopes.

\begin{corollary}\label{cor:n2}
  Let $G \subset \GL_2(C)$ is an algebraic group.
  Then there exists a \hyperlink{def_envelope}{toric envelope} $H$ for $G$ containing $Z_2$ and satisfying one of the following
  \begin{multicols}{2}\raggedcolumns
  \begin{itemize}[itemsep=0pt,topsep=1pt]
    \item $\dim H = 4$ and $\deg H = 1$;
    \item $\dim H = 3$ and $\deg H = 1$;
    \item $\dim H = 2$ and $\deg H = 2$;
    \item $\dim H = 2$ and $\deg H = 1$;
    \item $\dim H = 1$ and $\deg H \leqslant 60$.
  \end{itemize}
  \end{multicols}
\end{corollary}

\section{Proof of Theorem~\ref{thm:n3}}
\label{sec:pf4}
\begin{proof}
By~\cite[Theorem~4, p. 286]{VinbergOnishik}, $G$ can be written as a semidirect product $G_0 \ltimes U$, where $U$ is the unipotent radical of $G$ and $G_0$ is a reductive subgroup of $G$.
   Since $G_0$ is reductive, its representation in $C^3$ is completely reducible (see~\cite[Theorem~4.3, p. 117]{Hochschild}).
  There are three cases for the dimensions of $G_0$-irreducible components of $C^3$.
  \begin{enumerate}[label=(\alph*), itemsep=0pt, topsep=2pt]
     \item\label{case:1_1_1} $C^3 = W_1 \oplus W_2 \oplus W_3$, where $\dim W_1 = \dim W_2 = \dim W_3 = 1$ and $W_1, W_2$, and $W_3$ are irreducible $G_0$-representations;
    
    \item\label{case:2_1} $C^3 = V_1 \oplus V_2$, where $\dim V_1 = 1$, $\dim V_2 = 2$, and $V_1$ and $V_2$ are irreducible $G_0$-representations;
    
    \item\label{case:3} $C^3$ is $G_0$-irreducible.
  \end{enumerate}
  
  \paragraph{Case~\ref{case:1_1_1}: $C^3 = W_1 \oplus W_2 \oplus W_3$,} where $\dim W_1 = \dim W_2 = \dim W_3 = 1$ and $W_1, W_2$, and $W_3$ are irreducible representations of $G_0$.
  Then $G_0$ is diagonalizable in some basis, we fix such a basis.
    Let $D$ be the group of all diagonal matrices in the basis. We consider
    $H := (D \cap N(U))^\circ \cdot G$.
    Since $D \cap N(U)$ commutes with $G_0$ and normalizes $U$, it normalizes $G$.  So $H$ is an algebraic group.
    Hence, $H$ is a toric envelope of $G$.
    
    Since $\deg D = 1$ and $3 \geqslant \dim (D \cap N(U)) \geqslant \dim Z_3 = 1$, Corollary~\ref{cor:intersect_with_normalizer} implies that $\deg (D \cap N(U)) \leqslant 9$.
    Lemma~\ref{lem:unipotent_bound} implies that $\deg U \leqslant 2$.
    Since $D \cap N(U)$ is reductive, Lemma~\ref{lem:levi_bound} implies that
    \[
    \deg (D \cap N(U)) \cdot U \leqslant 9 \cdot 2 \cdot 8 = 144.
    \]
    Then 
    \underline{$\deg H \leqslant \deg (D \cap N(U))G = \deg (D \cap N(U)) U \leqslant 144$}.
  
  \paragraph{Case~\ref{case:2_1}: $C^3 = V_1 \oplus V_2$}, where $\dim V_1 = 1$, $\dim V_2 = 2$, and $V_1$ and $V_2$ are irreducible representations of $G_0$.
   According to the presentations~\eqref{eq:diag_mat} and~\eqref{eq:triang_mat} for $G_0$ and $U$ constructed in the proof of Lemma~\ref{lem:levi_bound}, one of the following two cases holds:
  \begin{gather*}
   \text{either }
   G_0 \subset \bigl\{ \begin{pmatrix} A & 0 \\ 0 & b \end{pmatrix} \;\mid\; A \in \GL_2(C), b \in C^\ast \bigr\} \text{ and } U \subset \bigl\{ \begin{pmatrix}
   1 & 0 & u \\
   0 & 1 & v \\
   0 & 0 & 1
   \end{pmatrix} \;\mid\; u, v \in C \bigr\},\\
   \text{ or }
   G_0 \subset \bigl\{ \begin{pmatrix} b & 0 \\ 0 & A \end{pmatrix} \;\mid\; A \in \GL_2(C), b \in C^\ast \bigr\} \text{ and } U \subset \bigl\{ \begin{pmatrix}
   1 & u & v \\
   0 & 1 & 0 \\
   0 & 0 & 1
   \end{pmatrix} \;\mid\; u, v \in C \bigr\}
   \end{gather*}
   We will consider the former. The latter is completely analogous.
   
   Let $T_1$ be the group of the matrices of the form $\operatorname{diag}(1, 1, a)$, where $a \in C^\ast$.
   Since $T_1 \subset Z(G_0)$, $T_1G_0$ is an algebraic group.
   Moreover, there is a decomposition $G_0 = T_1 G_{1}$, where $G_1$ acts trivially on $V_1$.
   Then $G_1$ can be considered as a subgroup of $\GL(V_2) \cong \GL_2(C)$.
   Let $H_1$ be a toric envelope for $G_1$ given by Corollary~\ref{cor:n2}.
   Then there exists a torus $T_2 \subset \GL(V_2)$ such that $H_1 = T_2G_1$.
   Then $T_2$ and $T_1$ commute, so $T := T_1T_2$ is a torus.
   Thus, $H_0 := TG_0$ is a toric envelope for $G_0$.
   
   We set $H := (H_0 \cap N(U)) \cdot U$.
   Since $G_0 \subset N(U)$, $H = (T\cap N(U)) G_0U$, so $H$ is a toric envelope of $G = G_0U$.
   Using Lemma~\ref{lem:levi_bound} with more precise bound given by~\eqref{eq:prod_precise}, we obtain
   \begin{equation}\label{eq:bound2_1}
   \deg H \leqslant 4 \deg (H_0 \cap N(U)) \deg U.
   \end{equation}
   Direct computation shows that any subgroup of 
   \[
   \left\{ \begin{pmatrix}
   1 & 0 & u \\
   0 & 1 & v \\
   0 & 0 & 1
   \end{pmatrix} \;\mid\; u, v \in C \right\}
   \]
   is an affine subspace of $\Mat_3(C)$, so $\deg U = 1$.
   Both $T_1$ and $Z_2 \subset H_1$ normalize $U$, so $\dim (H_0 \cap U) \geqslant 2$.
   Using Corollary~\ref{cor:intersect_with_normalizer} and classification from Corollary~\ref{cor:n2}, we obtain
   \[
   \deg (H_0 \cap N(U)) \leqslant 3^{\dim H_1 - 1}\deg H_1 \leqslant \max(27 \cdot 1, 9 \cdot 1, 3 \cdot 2, 1 \cdot 60) = 60.
   \]
   Plugging all the bounds into~\eqref{eq:bound2_1}, we obtain \underline{$\deg H \leqslant 240$}.
    
    \paragraph{Case~\ref{case:3}: $C^3$ is $G_0$-irreducible.}
    Since the space of fixed vectors of $U$ is $G_0$-invariant, it coincides with $C^3$, so $U = \{e\}$.
    \cite[Proposition, p. 181]{Borel} implies that $G_0^\circ$ can be written as $S T$, where $T := C(G_0^\circ)^\circ$ is torus and $S := [G_0^\circ, G_0^\circ]$ is semisimple.
    
    If $C^3$ is an irreducible $S$-representation, let $H := Z_3\cdot G$.
    As in the proof of Lemma~\ref{lem:connected_reductive} (see~\eqref{eq:whitehead_cor}), one can show that $N(S)^\circ = S Z(S)^\circ$.
    Schur's lemma implies that $Z(S) = Z_3$.
    Since $H \subset N(S)$ and $H^\circ \supset S \cdot Z(S)$, we obtain that $H^\circ = N(S)^\circ$, so $\deg H \leqslant \deg N(S)$.    
    Since $\dim N(S) \geqslant \dim Z_3 + \dim S \geqslant 4$, Corollary~\ref{cor:intersect_with_normalizer} applied with $G_0 = \GL_3(C)$ and $G_1 = S$ implies that \underline{$\deg H \leqslant 3^5 = 243$.}
    
    If $C^3$ is not an irreducible representation of $S$, then there exists an $S$-invariant one-dimensional subspace spanned by a vector $v$.
    Since $[S, S] = S$, $gv = v$ for every $g \in S$.
    Consider a subspace $I = \{u \in C^3 \;|\; \forall g \in S \; gu = u \}$.
    Since $v \in I$, $\dim I > 0$. 
    Since $S$ is normal in $G_0$, $I$ is $G_0$-invariant.
    Since $C^3$ is $G_0$-irreducible, $I = C^3$, so $S = \{e\}$.
    Since $U = S = \{e\}$, \cite[Lemma~10.10]{Wehrfritz} implies that $G = G_0 = \Gamma T$ for some finite $\Gamma$.

    Let $T_0$ be any maximal element of the set of all the tori containing $T$ and normalized by $\Gamma$.
    We set $H := T_0\cdot G$, then $H$ is a toric envelope of $G$.
    Since $T_0$ is a torus, all its irreducible representations are one-dimensional and are described by characters of $T_0$.
    We denote distinct characters of $T_0$ by $\chi_1, \ldots, \chi_s$. Then we write
    \[
    C^3 = \bigoplus\limits_{i = 1}^s V_{\chi_i}, \text{ where }  V_{\chi_i} := \{v \;|\; \forall g \in T_0 \; gv = \chi_i(g)v\} \neq \{0\} \text{ for } i = 1, \ldots, s.
    \]
        Since $\Gamma$ normalizes $T_0$, for every $1 \leqslant i \leqslant s$, there exists $1 \leqslant j \leqslant s$ such that $\Gamma(V_{\chi_i}) \subset V_{\chi_j}$.
    Consider possible values of $s$

      \paragraph{$\mathbf{s = 1}$.} Then $T \subset Z_3$, so $T_0 = Z_3$.
      By Lemma~\ref{lem:product}, we can further assume that $\Gamma \subset \SL_3(C)$.
      We will use the classification of finite subgroups of $\SL_3(C)$ from~\cite[p. 2-3]{YY} (see also~\cite[\S 3]{Serrano}).
      
      Since there is no torus strictly containing $Z_3$ that is normalized by $\Gamma$, $\Gamma$ is an imprimitive subgroup of $\SL_3(C)$ (see~\cite[p. 10]{YY}).
      Thus, only cases (E)-(K) from~\cite[p. 2-3]{YY} are possible. 
      One can see that every group $\Gamma$ of types (E)-(K) satisfies one of the following
      \begin{itemize}
        \item $|\Gamma| \leqslant 360$ (cases (E), (F), (H), (I), and (J));
        \item $|\Gamma| \leqslant 1080$ and $\Gamma$ contains a matrix $\omega I_3$, where $\omega$ is a primitive cubic root of unity and $I_3$ is the identity matrix (cases (J), (L), and (K)).
      \end{itemize}
      In both cases we have \underline{$\deg H \leqslant |\Gamma| / |\Gamma \cap Z_3| \leqslant 360.$}
      
      \paragraph{$\mathbf{s = 2}$.} Without loss of generality we can assume that $\dim V_{\chi_1} = 2$ and $\dim V_{\chi_2} = 1$.
      Then $V_{\chi_1}$ is an invariant subspace for both $\Gamma$ and $T_0$.
      This contradicts the assumption that $C^3$ is an irreducible $G_0$-representation.
      
      \paragraph{$\mathbf{s = 3}$.} If we choose a basis $e_1, e_2, e_3$ such that $e_i \in V_{\chi_i}$ for $1 \leqslant i \leqslant s$, every element of $G$ can be written in this basis as a product of a diagonal matrix and a permutation matrix.
      Let $D$ be a group of all diagonal matrices in this basis. 
      Then $\Gamma$ normalizes $D$ and $T_0 \subset D$, so $T_0 = D$.
      Thus, $H := \Gamma D$.
      Since the number of connected component of $H$ does not exceed the number of permutation matrices and $\deg D = 1$, we have \underline{$\deg H \leqslant 3! = 6$}. 
  
  \smallskip
  \noindent
  In all the cases above, we constructed a toric envelope $H$ of $G$ such that $\deg H \leqslant 360$.
\end{proof}


\paragraph{Acknowledgements.} The authors are grateful to Ivan Arzhantsev, Anton Baikalov, Harm Derksen, Gregor Kemper, Alexey Ovchinnikov, Michael Singer, Jacques-Arthur Weil, and the referees for their suggestions and helpful discussions.


\bibliographystyle{amsplain}
\bibliography{bibdata}

\end{document}